\documentclass[11pt,reqno]{amsart}

\usepackage{a4wide}
\usepackage{mathtools}
\mathtoolsset{showonlyrefs}
\usepackage{tikz}
\tikzstyle{every picture}=[line cap=round,line join=round,line width=.7pt,minimum size=3pt,every label/.append style={font=\small}, label distance=2pt]
\tikzstyle{EDR}=[draw=red,line width=1pt,preaction={clip, postaction={pattern=north west lines, pattern color=red}}]
\tikzstyle{EDB}=[draw=red,line width=1pt,preaction={clip, postaction={pattern=north west lines, pattern color=blue}}]
\tikzstyle{EDG}=[draw=red,line width=1pt,preaction={clip, postaction={pattern=north west lines, pattern color=green}}]
\usepackage{caption}
\usepackage{color}

\usepackage{graphicx}
\usepackage{amssymb,amsmath,amsthm,mathrsfs,xfrac}
\usepackage{enumerate,color}
\usepackage{hyperref}
\usepackage{array}
\usepackage{amsfonts}
\usepackage{psfrag}
\usepackage{soul}
\usepackage{pgf,tikz}
\usetikzlibrary{arrows}
\usetikzlibrary{patterns}
\usepackage[title]{appendix}
\usepackage{soul}
\newtheorem{theorem}{Theorem}
\newtheorem*{theoreme}{Main Theorem}

\newtheorem{lemma}{Lemma}

\theoremstyle{definition}
\newtheorem{example}{Example}
\newtheorem{remark}{Remark}

\numberwithin{equation}{section}
\allowdisplaybreaks[4]

\renewenvironment{proof}{\smallskip\noindent\emph{\textbf{Proof.}}%
  \hspace{1pt}}{\hspace{-5pt}{\nobreak\quad\nobreak\hfill\nobreak%
    $\square$\vspace{2pt}\par}\smallskip\goodbreak}

\newenvironment{proofof}[1]{\smallskip\noindent{\textbf{Proof~of~#1.}}%
  \hspace{1pt}}{\hspace{-5pt}{\nobreak\quad\nobreak\hfill\nobreak%
    $\square$\vspace{2pt}\par}\smallskip\goodbreak}

\newcommand{\bbar}[1]{\bar{\bar{#1}}}

\newcommand{\modulo}[1]{{\left|#1\right|}}

\newcommand{\R}{\mathbb{R}}

\newcommand{\N}{\mathbb{N}}

\renewcommand{\epsilon}{\varepsilon}
\newcommand{\e}{\varepsilon}
\renewcommand{\phi}{\varphi}
\renewcommand{\theta}{\vartheta}

\renewcommand{\div}{\mathinner{\mathop{{\rm div}}}}
\newcommand\red[1]{{\color{red}#1}}

\newcommand{\eps}{\varepsilon}

\newcommand{\A}{\mathbb{A}}

\newcommand{\be}{\begin{equation}}
\newcommand{\ee}{\end{equation}}

\definecolor{ffqqqq}{rgb}{1.,0.,0.}
\definecolor{uuuuuu}{rgb}{0.26666666666666666,0.26666666666666666,0.26666666666666666}

\makeatletter
\let\@fnsymbol\@arabic
\makeatother

\graphicspath{{./figure/}}

\title{Blow up for nonlinear wave-type equations with perturbed derivatives}

\author{F. A. Chiarello} 
\address[Felisia Angela Chiarello]{\newline DISIM, University of L'Aquila, Via Vetoio Coppito ED. 1, 67100, L'Aquila, Italy. }
\email[]{felisiaangela.chiarello@univaq.it}

\author{G. Girardi}
\address[Giovanni Girardi]{\newline
Department of Industrial Engineering and Mathematical Sciences, Polytechnic University of Marche, Via Brecce Bianche 12, 60128, Ancona, Italy.}
\email[]{g.girardi@univpm.it}

\author{S. Lucente}
\address[Sandra Lucente]{\newline
Dipartimento Interateneo di Fisica, University of Bari, Via Orabona n. 4, 70126, Bari, Italy.}
\email[]{sandra.lucente@uniba.it}



\date{}

\begin{document}
\baselineskip16pt

\maketitle

\begin{abstract}
We investigate semilinear wave-type equations that can be recast as wave equations with derivatives perturbed by zero-order terms. This framework covers several well-studied cases, including the scale-invariant wave equation. In this setting, we refine existing blow-up results for radial initial data with suitable decay, and identify conditions on the zero-order terms that govern the interplay between derivative perturbations, initial data size, and nonlinearity exponent.
 \end{abstract}

% {\bf Keywords:} {Scale invariant damped wave, slow decay, lifespan} 
 
% {\bf MSC2010} Primary 	35B33;
%Secondary 35L70.

\section{Introduction}

In this paper we treat the blow-up for nonlinear wave type equation with perturbed derivatives. 

\subsection{Main Result}

We consider a vector field $A:=( A_1,\dots A_n)\in C^1(\R^n,\R^n)$ and $A_0\in C^1(\R,\R)$. 
For all $i=1, \dots n$ we introduce the following partial derivative with respect to $x_i$:
\begin{equation}\label{cov_space}
\tilde{\partial}_{x_i}:=\partial_{x_i}+A_i(x).
\end{equation}
In particular, the second partial derivative is given by
\[\tilde \partial_{x_ix_i}=(\partial_{x_i}+ A_i(x))(\partial_{x_i}+A_i(x)),\]
and the corresponding Laplace operator is $\tilde \Delta =\tilde \partial_{x_1x_1}+\dots \tilde \partial_{x_nx_n} $.\\
Similarly, given $ A_0:\R\to \R$ we introduce the following partial derivative with respect to $t$:
\begin{equation}\label{cov_time}
\tilde{\partial}_t:=\partial_t+  A_0(t);
\end{equation}
then, $\tilde\partial_{tt}:=(\partial_t+  A_0(t))(\partial_t+ A_0(t))$. \\
We set $\A:=(A_0,A)$ and we call $\A$-derivative the perturbed derivatives in \eqref{cov_space} or \eqref{cov_time}.

Let $h\in \mathcal{C}^1([0,\infty)\times \R^n,\R)$ and $F\in \mathcal{C}^1(\R,\R)$ positive functions such that\begin{equation}\label{eq:F_up}
F(s)\geq |s|^p,
\end{equation}
for some $p>1$.
Given $\varepsilon>0,$ we consider the Cauchy problem
\begin{equation}
\label{eq:covariant_wave_equation*}
\begin{cases}
\tilde\partial_{tt}v(t,x)- \tilde\Delta v(t,x)=h(t,x)v(t,x)+F(\tilde \partial_t^j v(t,x)), \quad (t,x)\in [0,\infty)\times \R^n, \\
v(0,x)=0,&\\
v_t(0,x)= \varepsilon v_1(x),&
\end{cases}
\end{equation}
with $j\in \{0,1\}$ and we investigate blow-up results and lifespan estimates with decaying initial data. \\
Our main assumption is that the vector field $A$ is conservative, namely there exists $U\in C^2(\R^n,\R)$ such that
\begin{equation}
\label{eq:potential}
    \nabla U(x)= A(x), \quad \text{ for all } x\in \R^n;
\end{equation}
moreover, we denote by $G:[0,\infty)\to \R$ a differentiable function satisfying 
\begin{equation}
\label{eq:Gdef}
G'(t)=A_0(t), \quad \text{for any } t\geq 0.
\end{equation}
We will prove the following result.
\begin{theoreme}
Let $n\geq 2$. Let us consider $v_1$ a radial smooth function satisfying 
\begin{equation}
    v_1(r)\geq \frac{M}{(1+r)^{\alpha+1}}, \quad \forall r>0,
\end{equation}
for some $\alpha>-1$ and $M>0.$ Assume that 
\begin{equation}
    \lim_{t\to+\infty} \left[\Big(\frac{2-j}{p-1}-\alpha\Big)\log(t)-\max_{s\in [0,t]}G(s)
    +\frac{1}{p} \min_{\xi\in [0,3t]} U\left(\xi\right) -\max_{\xi\in [0,3t]} U\left(\xi\right)\right]=+\infty,
\end{equation}
then the classical solution $v$ of \eqref{eq:covariant_wave_equation*}  with $j=0,1$ and radial $h,U,v_1$ blows up.
\end{theoreme}
\subsection{Motivation for using the perturbed derivatives}

\label{sec:motivations}

Many well-known equations in literature can be rewritten as a wave-type equation with perturbed derivatives properly choosing suitable functions $A_0, A_1, \dots A_n$ in the definition of $\A$-derivatives.
The very general form is 
\begin{equation}\label{eq:general}
\tilde\partial_{tt} v-\tilde\Delta v=F,
\end{equation}
where 
\begin{equation}
\tilde{\partial}_t:=\partial_t+  A_0(t,x), \quad \tilde{\partial}_{x_i}:=\partial_{x_i}+A_i(t,x), \quad i=1,\dots, n.
\end{equation}

Assume, for instance, $A_0=0$ and $A:=( A_1,\dots A_n)\in C^1(\R\times \R^n,\R^n)$. 
In this case, the wave-type equation in problem \eqref{eq:general} reads as
\begin{equation} 
\label{eq:general_A0=0}
\partial_{tt}v-\Delta v - 2 A\cdot \nabla v - (\div_x(A)+|A|^2)v=F,
\end{equation}
that is a wave equation with lower order terms. It is well-known that the presence of lower order terms can influence well-posedness results for the corresponding Cauchy problem. 

Moreover, it would be interesting to consider the case of complex valued $\A$; in particular, if $A_0=0$ and $A$ is purely imaginary with $\div(A)=0$ then equation \eqref{eq:general_A0=0} corresponds to a wave equation with electromagnetic potential and Coulomb gauge.
Many papers treat this equation by obtaining linear estimates towards some global existence or scattering results (see \cite{DA2020,DAF2008, FV2009, GYZZ2022}). Up to our knowledge, no blow up result has been obtained in this setting, for this reason we started from the simplest case of real perturbation of the derivatives. 

Considering the opposite case $A_0=A_0(t,x)$ and $A=0$, we gain 
\[ 
\partial_{tt}v-\Delta v +2 A_0 \partial_t v + (\partial_t A_0 +A_0^2) v=F,
\]
that is, a wave equation with damping and mass term. Conversely, if we take a wave equation with damping and a mass term of the form 
\[
\partial_{tt}v-\Delta v +b(t) v_t +m^2(t) v=F;
\]
this can be rewritten in the form of $\mathbb A$-derivatives if and only if 
\[
m^2(t)=\frac{1}{2}\left(b'(t)+\frac{b^2(t)}{2}\right),
\]
that is, the possibility to write the equation using the $\mathbb A$-derivatives is strictly connected with an interaction among the low-order terms. 
On the other hand, if $b(t)=\frac{\mu}{2(1+t)}$, then we are treating the scale-invariant wave equation that has been very well studied in the last decade starting from the suitable change of variable that preserves the wave shape of the equation (see \cite{DALR2015, NPR2017, PT2021}).

In general, the introduction of the $\A$-derivatives allows one to gain some useful shape-preserving properties; in particular, if $A_0(t)=-\beta'(t)/\beta(t)$ for some $\beta\in C^1([0,\infty),\R)$, then one can easily see that
\[\tilde\partial_t (\beta(t)f(t))=\beta(t)\partial_t f(t)\]
for any $f\in C^1([0,\infty),\R).$
The application of this identity, together with a suitable change of coordinates, enable us to prove that our non-linear wave-type problem with perturbed derivatives \eqref{eq:covariant_wave_equation*} is equivalent to a classical wave equation with a perturbed non-linearity. 
For this reason, we choose to perturb the derivatives splitting the variables as in definitions \eqref{cov_space} and \eqref{cov_time}, that is $A_0$ depends only on $t$-variable and $A$ depends only on $x$-variables. In this case the following lemma holds true:
\begin{lemma}
\label{lem:change_of_coordinates}
Let us assume that $v=v(t,x)$ solves the equation \eqref{eq:covariant_wave_equation*}
where  $\tilde\partial_{x_i}$ and $\tilde\partial_t$ are defined in \eqref{cov_space}  and \eqref{cov_time} respectively; let $U\in C^2(\R^n,\R)$ and $G\in C^2([0,\infty),\R)$ given by equations \eqref{eq:potential} and, respectively, \eqref{eq:Gdef}. Then, the function 
\begin{equation*}
u(t,x)=e^{G(t)} e^{U(x)} v(t,x),
\end{equation*}
solves the Cauchy problem
\begin{equation}
\begin{cases}
\label{eq:wave_equation*}
\partial_{tt}u(t,x)-\Delta u(t,x)= h(t,x)u(t,x)+e^{G(t)} e^{U(x)} F\Big(e^{-G(t)} e^{-U(x)}\partial_t^j u(t,x)\Big),&\\
u(0,x)=0,&\\
u_t(0,x)=\varepsilon e^{U(x)} v_1(x),&
\end{cases}
\end{equation}
for all $(t,x)\in [0,\infty)\times \R^n$.
\end{lemma}
\begin{proof}
It holds 
\begin{equation*}
\partial_t u(t,x)=  A_0(t) u(t,x)+e^{G(t)} e^{U(x)} \partial_t v(t,x)=e^{G(t)} e^{U(x)} \tilde\partial_t v, 
\end{equation*}
and then,
\begin{align*}
\partial_{tt} u(t,x)&=  A_0'(t)u(t,x)+ A_0(t) \partial_t u(t,x)+ A_0(t)e^{G(t)} e^{U(x)} \partial_t v(t,x)+e^{G(t)} e^{U(x)}\partial_{tt}v(t,x)\\
&=  A_0'(t)u(t,x) +  A_0(t)^2 u(t,x)+2  A_0(t) e^{G(t)} e^{U(x)}  \partial_t v(t,x) +e^{G(t)} e^{U(x)}\partial_{tt}v(t,x).
\end{align*}
Moreover, for all $i=1, \dots, n$ we have
\begin{equation*}
\partial_{x_i} u= e^{G(t)}e^{U(x)}(A_i(x)v(t,x)+\partial_{x_i}v(t,x)),
\end{equation*}
and then,
\begin{align*}
\partial_{x_ix_i} u= e^{G(t)}e^{U(x)}(A_i(x)^2v(t,x)+\partial_{x_i} A_i(x)v(t,x)+2 A_i(x)\partial_{x_i}v(t,x)+\partial_{x_ix_i}v(t,x));
\end{align*}
we conclude that
\begin{equation}
\label{eq:wave_resolution*}
\begin{aligned}
\partial_{tt}u(t,x) &- \Delta u(t,x)=  A_0'(t)u(t,x) +  A_0(t)^2 u(t,x)+2  A_0(t) e^{G(t)}e^{U(x)} \partial_t v(t,x)\\& 
-e^{G(t)}e^{U(x)}\Big(v(t,x)| A(x)|^2+\div(A(x))v(t,x)+2  A(x)\cdot \nabla v(t,x)  \Big) \\
& +e^{G(t)} e^{U(x)}(\partial_{tt}v(t,x)-\Delta v(t,x)),
\end{aligned}
\end{equation}
where $\nabla:=(\partial_{x_1},\dots, \partial_{x_n}$).
On the other hand, it holds 
\begin{align*}
\tilde \partial_{tt}v(t,x)&=(\partial_t+ A_0(t))(\partial_t v(t,x)+ A_0(t)v(t,x))\\
& =\partial_{tt}v(t,x)+ A_0'(t)v(t,x)+2 A_0(t)\partial_t v(t,x)+ A_0(t)^2 v(t,x),
\end{align*}
and similarly, 
\begin{align*}
\tilde \partial_{x_i x_i}v(t,x)=\partial_{x_i x_i}v(t,x)+\partial_{x_i}A_i(x)v(t,x)+2A_i(x)\partial_{x_i}v(t,x)+ A_i(x)^2 v(t,x),
\end{align*}
for all $i=1,\dots n$; we get
\begin{equation*}
\begin{aligned}
\tilde \partial_{tt}v(t,x)-\tilde \Delta v(t,x)&= \partial_{tt}v(t,x)+ A_0'(t)v(t,x)+2 A_0(t)\partial_t v(t,x)+ A_0(t)^2 v(t,x) \\
& -\Delta v(t,x)-v(t,x)\div(A(x))-2A (x)\cdot \nabla v(t,x)-v(t,x)|A (x)|^2 .
\end{aligned}
\end{equation*}
In particular, $v=v(t,x)$ solves equation \eqref{eq:covariant_wave_equation*} if and only if, 
\begin{equation}
\label{eq:covariant_wave_resolution*}
\begin{aligned}
\partial_{tt}v(t,x)-\Delta v(t,x)=&h(t,x)v(t,x)+F(\tilde \partial_t^j v(t,x))\\
&- A_0'(t)v(t,x)-2 A_0(t)\partial_t v(t,x)- A_0(t)^2 v(t,x) \\
&+v(t,x)\div(A(x))+2A(x)\cdot \nabla v(t,x)+v(t,x)|A(x)|^2.
\end{aligned}
\end{equation}
Repleacing \eqref{eq:covariant_wave_resolution*} in \eqref{eq:wave_resolution*} we get
\begin{equation*}
\begin{aligned}
\partial_{tt}u(t,x) - \Delta u(t,x)=
%& \red{ A_0'(t)u(t,x)} + \blue{ A_0(t)^2 u(t,x)}+\magenta{2  g(t) e^{G(t)} e^{U(x)} \partial_t v(t,x)}\\& 
%
%-e^{G(t)} e^{U(x)}(\brown{v(t,x)|A(x)|^2}+\div(A(x))v(t,x)+\gray{2 A(x)\cdot \nabla v(t,x)} )\\
%
%& +e^{G(t)} e^{U(x)}(h(t,x)v(t,x)+F(\tilde\partial_t^jv(t,x))- A_0'(t)v(t,x)-2 A_0(t)\partial_t v(t,x)\\
%&-\blue{ A_0(t)^2 v(t,x)}+\olive{v(t,x)\div(A(x))}+2A(x)\cdot \nabla v(t,x)+v(t,x)|A(x)|^2)\\
%
h(t,x)u(t,x)+e^{G(t)} e^{U(x)}F\Big(e^{-G(t)} e^{-U(x)}\partial_t^j u(t,x)\Big).
\end{aligned}
\end{equation*}
\end{proof}

\subsection{Radial assumption and critical exponent}
Many authors have studied existence theorems and blow up results for the following non-linear classical wave equation
\begin{equation}
\label{eq:classical_wave_equation}
\begin{cases}
\partial_{tt}v(t,x)- \Delta v(t,x)=|\partial_t^jv(t,x)|^p, \quad (t,x)\in [0,\infty)\times \R^n, \\
v(0,x)=0,&\\
v_t(0,x)= \varepsilon v_1(x).&
\end{cases}
\end{equation}
which corresponds to problem \eqref{eq:covariant_wave_equation*} with $\A\equiv 0$, $h\equiv 0$ and $j\in \{0,1\}$. In the case of compactly supported initial data the critical exponent for problem \eqref{eq:classical_wave_equation} is 
\begin{equation}
p_j(n)=\begin{cases}
\frac{n+1+\sqrt{n^2+10n-7}}{2(n-1)} \quad &\text{ if } j=0,\\
\frac{n+1}{n-1} \quad & \text{ if } j=1;
\end{cases}
\end{equation}
namely, problem \eqref{eq:classical_wave_equation} admits a global solution for \textit{small} initial data if $p>p_j(n)$, and all the solutions to \eqref{eq:classical_wave_equation} blow up in finite time if $1<p\leq p_j(n)$.  For $j=0$ the exponent $p_0(n)$ is known as \textit{Strauss exponent} and it is the positive root of the polynomial equation $(n-1)p^2-(n+1)p-2=0$; whereas for $j=1$ the exponent $p_1(n)$ is known as \textit{Glassey exponent}. A complete bibliography about global existence results can be found in \cite{GLS1997}. \\
It is well-known that if the initial data has noncompact
support, one can find blowing-up solutions even for $p>p_0(n)$; in particular, slow decay yields blowing-up solutions while rapid
decay assures global solution.\\

More in details, suppose that $v_1$ is radially symmetric and satisfies 
\begin{align}
\label{eq:decay_data_intro}
&v_1(x)\geq \frac{M}{(1+|x|)^{\alpha+1}}, \quad  M>0, \; \alpha>-1;
\end{align}
for $j=0$ in \cite{takamura1995} the author proved that any classical solution to \eqref{eq:classical_wave_equation} blows up in finite time for any $1<p\leq p_c(\alpha)$ where 
\[p_c(\alpha):=1+\frac{2}{\alpha}\] 
independently on the space dimension $n\geq 2$. Later we will refer to $p_c(\alpha)$ as  \textit{slowing decaying critical exponent} (see Section \ref{sec:examples}). Similarly, for $j=1$ in \cite{takamura1995} has been proven that any classical solution to \eqref{eq:classical_wave_equation} blows up in finite time for any $1<p\leq 1+1/(\alpha+1)$. 
Additionally, from \cite{takamura1995} we know that the maximal existence time $T(\eps)$ of classical solutions to \eqref{eq:classical_wave_equation} satisfies the following lifespan estimate 
\begin{equation}
T(\eps)\leq C \eps^{-\frac{p-1}{2-jp-(p-1)\alpha}},
\end{equation}
for some constant $C>0$ independent of $t$ and $\eps$. 
The proof of blow-up results in high space dimension $n\geq 4$ requires the assumption of radial symmetry of the initial data, which could be avoided in low space dimension $n=2,3$ due to the positivity of the solution to the free wave equation with positive initial data \cite{A1986, AT1992}. In our case, the $\A$-derivatives that perturb the wave operator introduce greater complexity to the problem. Therefore, we assume radial symmetry in every dimension. On the other hand, the radial assumption is also necessary in \cite{K97} and \cite{K98} to prove the existence of slowing decaying radial solutions to \eqref{eq:classical_wave_equation} with $j=0$ for every $p>p_c(\alpha)$. \\

As already mentioned in Section \ref{sec:motivations}, the scale-invariant equation can be rewritten in the form \eqref{eq:general}, and the results related to it in the slowly decaying case have been studied in \cite{CGL2021} and \cite{GirardiLucente2021}. More precisely, we considered the following problem
\begin{equation}
\label{eq:scale-invariant_wave_equation}
{\small
\begin{cases}
\partial_{tt}v(t,x)- \Delta v(t,x)+\frac{\mu}{(1+t)}v_t +\frac{\mu}{2}\left(\frac{\mu}{2}-1\right)\frac{1}{(1+t)^2}v=|\tilde\partial_t^jv(t,x)|^p, \quad (t,x)\in [0,\infty)\times \R^n, \\
v(0,x)=0,&\\
v_t(0,x)= \varepsilon v_1(x),&
\end{cases}}
\end{equation}
which corresponds to problem \eqref{eq:covariant_wave_equation*} with $A_0=\frac{\mu}{2(1+t)}$, $A\equiv 0$, $h\equiv 0$ and $j\in \{0,1\}$. \\
If $j=0$ and $v_1$ has compact support, then the presence of scale invariant damping and mass in problem \eqref{eq:scale-invariant_wave_equation} determine a shift of the critical exponent which becomes $p_0(n+\mu)$ (see \cite{PT2021} and the references therein for additional details). In \cite{CGL2021} we removed the compact support assumption for $v_1$ and assume that $v_1$ is a slowly decaying initial data satisfying \eqref{eq:decay_data_intro}, and we proved that any solution to \eqref{eq:scale-invariant_wave_equation} blows up in finite time for any $1<p\leq p_c(\alpha+\mu/2)$. This confirm the shift of the critical exponents. In addition we showed that the lifespan satisfies
\begin{equation*}
T(\eps)\leq C \eps^{-\frac{2(p-1)}{4-(\mu+2\alpha)(p-1)}},
\end{equation*}
for some constant $C>0$ independent of $\eps$. A detailed review about global existence results for problem \eqref{eq:scale-invariant_wave_equation} is provided in \cite{CGL2021}. \\
In \cite{PT2021} the authors considered the same scale-invariant damped wave equation with mass \eqref{eq:scale-invariant_wave_equation}, with classical derivative-type nonlinearity $|\partial_t u|^p$ in place of the perturbed derivative-type nonlinearity $|\tilde\partial_t u|^p$; for small initial data $v_1$ with compact support, they proved that any solution blows up in finite time for any $1<p\leq p_1(n+\mu)$. In \cite{GirardiLucente2021} we considered problem \eqref{eq:scale-invariant_wave_equation} with a perturbed derivative-type non-linearity, namely $|\tilde\partial_t u|^p$ in place of $|\partial_t u|^p$; assuming $v_1$ to be a slowly decay function satisfying \eqref{eq:decay_data_intro} we proved that any classical solution blows up in finite time for any $1<p\leq p_c(2+2\alpha+\mu)$ and the maximal existence time satisfies
\begin{equation}
T(\eps)\leq C \eps^{-\frac{2(p-1)}{4-2p-(\mu+2\alpha)(p-1)}}.
\end{equation}
In this work, we adopt a much broader perspective given by the study of problem \eqref{eq:covariant_wave_equation*}, and we recover, as a confirmation, the results in \cite{CGL2021} and \cite{GirardiLucente2021}. \\
Although it falls outside the scope of this work, which focuses on blow-up with slowly decaying data, we would like to recall that the case of the scale-invariant wave equation for quasilinear equations was considered in \cite{GeorgievLucente2021}, where the corresponding $\A$-derivatives appeared at second order.

\subsection{Plan of the paper} 
In Section \ref{sec:blow-up} we provide a detailed version and the proof of the Main Theorem. We also show five examples of applications to emphasize the novelties of our results with respect to the existing literature. 
In Section \ref{sec:lifespan} we give a lifespan estimate in some particular examples that includes a generalization of \cite{CGL2021}.

\section{Blow-up results}
\label{sec:blow-up}
In this section we assume that $v_1(x)=v_1(\modulo{x})$, $h(t,x)=h(t,\modulo{x})$ and $U(x)=U(\modulo{x}).$ We set $r:=\modulo{x}.$\  The radial assumption gives us the possibility to apply the following crucial lemmas from \cite{takamura1995}.
\begin{lemma}\label{lemma:Tdeltam}
Let $n\geq 2$ and $m=[n/2]$. Given a smooth function $\psi=\psi(r)$, we 
denote by $u^{0}(t,r)$ the solution of the free wave problem 
\begin{equation}\label{eq:linear_radial}
\begin{cases}
    u^0_{tt}-u^0_{rr}-\frac{n-1}{r}u^0_r=0,  \quad & (t,r)\in [0,\infty)\times (0,\infty),\\
   u^0(0,r)=0, \\ u^0_t(0,r)= \psi(r)\,.
    \end{cases}
\end{equation}
Let $u=u(t,r)$ be a solution to the corresponding non-linear Cauchy problem 
\begin{align}\label{eq:comparisonlemma}
\begin{cases}
    u_{tt}-u_{rr}-\frac{n-1}{r}u_r =H(t,r,u,\partial_t u),  \quad & (t,r)\in [0,\infty)\times (0,\infty),\\
    u(0,r)=0, \\  
    u_t(0,r)= \psi(r)\,. 
    \end{cases}
\end{align}
If $H$ is a $C^1$ nonnegative function, then there exists a constant $\sigma_n> 0$ such that 
\begin{align}\label{eq:um}
u(t,r)&\geq  u^0(t,r)+\frac{1}{8r^m}\int_0^t d\tau \int_{r-t+\tau}^{r+t-\tau}\lambda^m H(\tau,\lambda,u(\tau,\lambda), \partial_t u(\tau,\lambda))d \lambda,\\
\label{eq:um0}
u^0(t,r)&\geq \frac{1}{8r^m} \int_{r-t}^{r+t} \lambda^m \psi(\lambda) d\lambda,
\end{align}
provided
\begin{equation}
\label{eq:r-t>}
    r-t\geq \sigma_n t>0.
\end{equation}
\end{lemma}
 \begin{lemma}
 \label{lem:comparison}
 Let $n\geq 2$ and $u$ be a solution to \eqref{eq:comparisonlemma}
with $\psi>0$. Assume that $H$ is a $C^1$ function satisfying $H(t,r,s,z)\geq 0$ for any $s,z\geq0$; then, we have $u(t,r)\geq 0$ for any $(t,r)$ such that \eqref{eq:r-t>} holds true.
 \end{lemma}
The first version of our result is the following.
\begin{theorem}
\label{th.main}
Let $n\geq 2$ and $\sigma_n>0$ given by Lemma \ref{lemma:Tdeltam}. Let us consider $v_1$ a radial smooth function satisfying 
\begin{equation}\label{eq:data_assumption}
    v_1(r)\geq \frac{M}{(1+r)^{\alpha+1}}, \quad \forall r>0,
\end{equation}
for some $\alpha>-1$ and $M>0.$ Assume that 
\begin{equation}\label{eq:interaction}
    \lim_{t\to+\infty} \left[\Big(\frac{2}{p-1}-\alpha\Big)\log(t)-\max_{s\in [0,t]}G(s)
    +\frac{1}{p} \min_{r\in [\sigma_n t,(2+\sigma_n)t]} U(r) -\max_{r\in [\sigma_n t,(2+\sigma_n)t]} U(r)\right]=+\infty,
\end{equation}
then the classical solution $v$ of \eqref{eq:covariant_wave_equation*}  with j=0 and radial $h,U,v_1$ blows up.
\end{theorem}
\begin{theorem}\label{th.main.derivative}
Let $n\geq 2$ and $\sigma_n>0$ given by Lemma \ref{lemma:Tdeltam}. Let us consider $v_1$ a radial smooth function satisfying 
\begin{equation}\label{eq:data_assumption_derivative}
    v_1(r)\geq \frac{M}{(1+r)^{\alpha+1}}, \quad \forall r>0,
\end{equation}
for some $\alpha>-1$ and $M>0.$ Assume that 
\begin{equation}\label{eq:interaction_derivative}
    \lim_{t\to+\infty} \left[\Big(\frac{1}{p-1}-\alpha\Big)\log(t)-\max_{s\in [0,t]}G(s)
    +\frac{1}{p} \min_{r\in [\sigma_n t,(2+\sigma_n)t]} U(r) -\max_{r\in [\sigma_n t,(2+\sigma_n)t]} U(r)\right]=+\infty,
\end{equation}
then the classical solution $v$ of \eqref{eq:covariant_wave_equation*}  with j=1 and radial $h,U,v_1$ blows up.
\end{theorem}
\begin{remark}
\label{rem:sigma_n}
In the Main Theorem for easy of presentation we provided a weaker version of Theorem \ref{th.main} and Theorem \ref{th.main.derivative}; indeed, the constant $\sigma_n$ in Lemma \ref{lemma:Tdeltam} coincides with the constant $\delta_m$ in  \cite[Lemma 2.5]{takamura1995} which can be choosen in the interval $(0,1)$.
\end{remark}

\subsection{ Applying Theorem \ref{th.main}} 
\label{sec:examples}
    It is worth noticing that in condition \eqref{eq:interaction} there is interaction between the decay of the initial datum, the exponent of the nonlinear term and the growth of the potentials of the functions $A_0$ and $A$ appearing in our definition of time and space derivatives.  Let us show some examples in which this condition can be written in more explicit form, treating the logarithmic term in \eqref{eq:interaction} as the leading one. For this reason, we call \textit{slowly decaying critical exponent}
    \begin{equation}
    \label{eq:free_critical}
    p_c:=1+\frac{2}{\alpha},
    \end{equation}
    which firstly appears in \cite[Theorem 1.2]{takamura1995} with classical derivatives, that is 
    $A_i=0$ in \eqref{cov_space}, and $A_0\equiv 0$ in \eqref{cov_time}. We also say that time or space derivatives are \emph{effective} when they  change the critical exponent with respect to $p_c$.\\
    We provide several examples that are summarized in Example \ref{ex:4}. This choice allows us to present the various interactions in more detail.
\begin{example} 
\label{ex:scaleinvariant}
Denote by $\langle x \rangle = \sqrt{1+|x|^2}$. Let us consider the \textit{scale invariant} case $A_0(t)=\frac{\mu}{2(1+t)}$ and $A(x)=\frac{\eta}{2}\frac{x}{\langle x \rangle^2}.$ This means $G(t)=\frac{\mu}{2}\log(1+t)$ and $U(x)=\frac{\eta}{2}\log\langle x \rangle$.
Suppose that 
\begin{equation}\label{eq:interex}
\left(\frac{2}{p-1}-\alpha-\frac{\mu}{2}-\frac{\eta}{2}\left(1-\frac{1}{p}\right)\right)>0;
\end{equation}
then 
\[
\left(\frac{2}{p-1}-\alpha-\frac{\mu}{2}-\frac{\eta}{2}\left(1-\frac{1}{p}\right)\right)\log(t)\to +\infty;
\]
hence \eqref{eq:interaction} holds. Condition \eqref{eq:interex} gives a critical exponent that depends on the decay of the initial datum and  the functions appearing in our definitions \eqref{cov_space} and \eqref{cov_time} of space and time derivatives. 
In particular for $\eta=0$ we have the result contained in \cite{CGL2021}. In this case, time and space derivatives are \emph{effective} since their presence gives a blow up exponent different from $p_c$ in \eqref{eq:free_critical}.
\end{example}
\begin{example}
\label{ex:boundedcase}
Let us suppose that $U(x)$ is \textit{bounded}. An explicit example would be $A(x)=x\langle  x\rangle^{-\beta-1}$ with $\beta>1$, which corresponds to $U(x)=\frac{1}{1-\beta} \langle x\rangle^{-\beta+1}.$  \\
Consider $\mathfrak{m}$ and $\mathcal{M}$ such that $\mathfrak{m}\leq U(x)\leq \mathcal{M}.$ In this case we have that
\begin{align*}
\Big(\frac{2}{p-1}-\alpha\Big)\log(t)&-\max_{[0,t]}G(\xi)
    +\frac{1}{p} \min_{[\sigma_n t,(2+\sigma_n)t]} U(\xi) -\max_{[\sigma_n t,(2+\sigma_n)t]} U(\xi) \\
 &   \geq \Big(\frac{2}{p-1}-\alpha\Big)\log(t)-\max_{[0,t]}G(\xi)
    +\frac{1}{p} \mathfrak{m}-\mathcal{M}. 
    \end{align*}
Being $\displaystyle{\max_{\xi\in [0,t]}G(\xi)}\geq G(0)=0$, the condition \eqref{eq:interaction} holds provided
\begin{equation}
\label{eq:interaction_bounded}
    \lim_{t\to+\infty} \left[\Big(\frac{2}{p-1}-\alpha\Big)\log(t)-\max_{[0,t]}G(\xi)\right]
    =+\infty.
\end{equation}
Assuming in addition $A_0\geq 0$, if there exists $\displaystyle{\lim_{s\to +\infty}}A_0(s)s=0$ (for example $A_0\in L^1([0,\infty))$), then for any $\varepsilon>0$ it holds $\displaystyle{\max_{[0,t]}G(\xi)\geq G(t)\geq \varepsilon\log(t)} $ for large $t\geq 0$. This means that condition \eqref{eq:interaction_bounded} is verified for any $p<1+2/\alpha=p_c$. 
Whereas, if $\displaystyle{\lim_{s\to +\infty}}A_0(s)s=\gamma>0,$ then we require that  $\left(\frac{2}{p-1}-\left(\alpha+\gamma\right)\right)>0$, that is 
\[p<1+\frac{2}{\alpha+\gamma}<p_c\,.\]
Summarizing, if $U$ is bounded,  only the time derivative can be \emph{effective}.
\end{example}
\begin{example}\label{ex:3}
If $\displaystyle{\lim_{\xi \to+\infty}}\frac{U(\xi)}{\log(\xi)}= 0.$ 
Then, for any $\varepsilon>0$ arbitrarily small and $t$ sufficiently large, we may estimate
\begin{align*}
\frac{1}{p} &\min_{[\sigma_n t,(2+\sigma_n)t]} U(\xi) -\max_{[\sigma_n t,(2+\sigma_n)t]} U(\xi) \\
 &   \geq 
    \frac{\varepsilon }{p} \min_{[\sigma_n t,(2+\sigma_n)t]} (-\log (\xi)) - \varepsilon \max_{[\sigma_n t,(2+\sigma_n)t]} \log (\xi) \\
 &   = 
    -\frac{\varepsilon }{p} \log ((2+\sigma_n) t) - \varepsilon \log ((2+\sigma_n)t)\\
& = -\varepsilon\Big(\frac{1}{p}+1\Big)\log(t)- \log\left((2+\sigma_n)^{-\varepsilon(1+\frac{1}{p})}\right).
    \end{align*}
Proceeding as in Example \ref{ex:boundedcase} we get
condition \eqref{eq:interaction} provided \eqref{eq:interaction_bounded}. Again, only the time derivative is \emph{effective} when $\displaystyle{\lim_{s\to +\infty}}A_0(s)s=\gamma>0$, with blow up exponent $p<1+2/(\alpha+\gamma)<p_c$.
\end{example}
\begin{example}\label{ex:4}
Now, we consider the case $\displaystyle{\lim_{\xi\to +\infty}\frac{U(\xi)}{\log(\xi)}}=\ell>0$.
For any $\varepsilon\in (0,\ell)$ arbitrarily small and $t$ sufficiently large, we may estimate
\begin{align*}
\frac{1}{p} &\min_{[\sigma_n t,(2+\sigma_n)t]} U(\xi) -\max_{[\sigma_n t,(2+\sigma_n)t]} U(\xi) \\
 &   \geq 
    \frac{(\ell-\varepsilon) }{p} \min_{[\sigma_n t,(2+\sigma_n)t]} \log (\xi) - (\ell+\varepsilon) \max_{[\sigma_n t,(2+\sigma_n)t]} \log (\xi) \\
 &   = 
    \frac{\ell-\varepsilon }{p} \log (\sigma_n t) - (\ell+\varepsilon) \log ((2+\sigma_n)t)\\
& = \left(\frac{\ell-\eps}{p}-\ell-\eps\right)\log(t)- \log\left((2+\sigma_n)^{(\ell+\varepsilon)}\right)+\log\left(\sigma_n^\frac{\ell-\varepsilon}{p}\right).
    \end{align*}
Being $\displaystyle{\max_{\xi\in [0,t]}G(\xi)}\geq G(0)=0$, the condition \eqref{eq:interaction} holds provided
\begin{equation}
\label{eq:interaction_bounded}
    \lim_{t\to+\infty} \left[\left(\frac{2}{p-1}-\alpha-\ell\left(1-\frac{1}{p}\right)\right)\log(t)-\max_{[0,t]}G(\xi)\right]
    =+\infty.
\end{equation}
Assuming $A_0\geq 0$, proceeding as in Example \ref{ex:boundedcase}, if  $\displaystyle{\lim_{s\to +\infty}}A_0(s)s=0$, then the space derivative is always effective. In particular, if $\alpha+\ell\leq 0$ one obtains the blow result for any $p>1$; whereas, if $\alpha+\ell> 0$ the solution to \eqref{eq:covariant_wave_equation*} blows up in finite time for any $1<p<p_\ell(\alpha)$ where $p_\ell(\alpha)$ is the unique root greater than one of the following 
\begin{equation}
(\alpha+\ell)p^2-(\alpha+2+2\ell)p+\ell=0.
\end{equation}
Analogously, if $\displaystyle{\lim_{s\to +\infty}}A_0(s)s=\gamma>0,$ then we require  
\begin{equation}
\left(\frac{2}{p-1}-\left(\alpha+\gamma\right)-\ell\left(1-\frac{1}{p}\right)\right)>0.
\end{equation}
This corresponds to condition \eqref{eq:interex} when $\gamma=\mu/2$ and $\ell=\eta/2$. In particular, the solution to \eqref{eq:covariant_wave_equation*} blows up in finite time for any $p>1$ if $\alpha+\gamma+\ell\leq 0$; whereas, if $\alpha+\gamma+\ell>0$ then condition \eqref{eq:interaction} holds for any $1<p<p_\ell(\alpha+\gamma)$.

We can conclude that Theorem \ref{th.main} introduces the important novelty of our new space derivative \eqref{cov_space} not considered in \cite{CGL2021}. But also when $\eta=0$ Theorem \ref{th.main} is stronger then the result in \cite{CGL2021} because it allows to consider a time derivative which leads to a \textit{non}-scale-invariant operator.

A similar result holds if $\displaystyle{\lim_{\xi\to +\infty}\frac{U(\xi)}{\log(\xi)}}=\ell<0$; indeed, the coefficient in the logarithmic term does not change: by the computation
\begin{align*}
\frac{1}{p} &\min_{[\sigma_n t,(2+\sigma_n)t]} U(\xi) -\max_{[\sigma_n t,(2+\sigma_n)t]} U(\xi) \\
&\hspace{4em} \geq  \left(\frac{\ell-\eps}{p}-\ell-\eps\right)\log(t)- \log\left(\sigma_n^{\ell+\varepsilon}\right)+\log\left((2+\sigma_n)^\frac{\ell-\varepsilon}{p}\right),
    \end{align*}
we arrive at \eqref{eq:interaction_bounded}.

It is worth noticing that if $\ell=0,$ we reduce to the case of Example \ref{ex:3}.

\end{example}
\begin{example} \label{ex:5}
    If $\displaystyle{\lim_{\xi \to+\infty}}\frac{U(\xi)}{\log(\xi)}= -\infty,$ then for any $M>0$ and $\xi>0$ sufficiently large, we may estimate $U(\xi)<-M \log(\xi)$, and then
\begin{align*}
\Big(\frac{2}{p-1}-\alpha\Big)\log(t)&-\max_{[0,t]}G(\xi)
    +\frac{1}{p} \min_{[\sigma_n t,(2+\sigma_n)t]} U(\xi) -\max_{[\sigma_n t,(2+\sigma_n)t]} U(\xi) \\
 \geq&
 \Big(\frac{2}{p-1}-\alpha\Big)\log(t)-\max_{[0,t]}G(\xi)+M\left(1-\frac{1}{p}\right) \log(\sigma_n t)\\
 \geq &\Big(\frac{2}{p-1}-\alpha+M\left(1-\frac{1}{p}\right)\Big)\log(t)-\max_{[0,t]}G(\xi)+M\left(1-\frac{1}{p}\right) \log(\sigma_n).
\end{align*}
Assuming in addition $A_0\geq 0$, if $A_0\in L^1([0,\infty))$ or $\displaystyle{\lim_{s\to +\infty}}A_0(s)s=\gamma\geq 0,$ then a blow up result occurs for any $p>1$. 
\end{example}

\begin{remark}
If $\displaystyle{\lim_{\xi \to+\infty}}\frac{U(\xi)}{\log(\xi)}= +\infty,$ we can notice that condition \eqref{eq:interaction} can never be satisfied. Indeed, for any $K>0$ and $t$ sufficiently large, we may estimate
\begin{align*}
\Big(\frac{2}{p-1}-\alpha\Big)\log(t)&-\max_{[0,t]}G(\xi)
    +\frac{1}{p} \min_{[\sigma_n t,(2+\sigma_n)t]} U(\xi) -\max_{[\sigma_n t,(2+\sigma_n)t]} U(\xi) \\ 
\leq &
 \Big(\frac{2}{p-1}-\alpha\Big)\log(t)-\max_{[0,t]}G(\xi)+\left(\frac{1}{p}-1\right)  \max_{[\sigma_n t,(2+\sigma_n)t]} U(\xi)\\
 \leq&
 \Big(\frac{2}{p-1}-\alpha\Big)\log(t)-\max_{[0,t]}G(\xi)+\left(\frac{1}{p}-1\right)  K \log((2+\sigma_n)t)\\
 \leq &\Big(\frac{2}{p-1}-\alpha+\frac{K}{p}-K\Big)\log(t)-\max_{[0,t]}G(\xi)+\left(\frac{1}{p}-1\right)  K \log(2+\sigma_n),
\end{align*}
then for all $p>1$ we can choose $K>0$ sufficiently large such that 
$\left(\frac{2}{p-1}-\alpha+\frac{K}{p}-K\right)<0.$ Being $\displaystyle{\max_{[0,t]} G(\xi)}\geq 0$ 
the left hand side is bounded from above preventing \eqref{eq:interaction}.
\end{remark}

\subsection{{ Proof of Theorem \ref{th.main}}}
Due to the radial assumptions, for $j=0$ we rewrite \eqref{eq:wave_equation*} as
\begin{equation}
\label{eq:radial_problem}
\begin{cases}
u_{tt}-u_{rr}-\frac{n-1}{r}u_r=H(t,r,u), \quad &t\geq0,\,r>0,\\
u(0,r)=0,  \\
u_t(0,r)=\varepsilon e^{U(r)} v_1(r),
\end{cases}
\end{equation}
where
\begin{equation} 
\label{eq:H}
H(t,r,u)=h(t,r)u+e^{G(t)} e^{U(r)}F\Big(e^{-G(t)} e^{-U(r)}u\Big).
\end{equation}
Focusing on the initial data in \eqref{eq:radial_problem}, we understand the interaction between the space derivatives and the initial data. Similarly, from \eqref{eq:H} we recognize the interaction between the nonlinear term and the time and space derivatives.

Coming back to our problem, let us consider $H$ defined by \eqref{eq:H}. Since $h, F\geq 0,$ then $H(t,r,s)\geq 0,$ for any $s\geq 0.$ Applying Lemma \ref{lem:comparison} this gives $u(t,r)\geq 0,$ provided \eqref{eq:r-t>} holds. In turn,  this gives \begin{equation}\label{eq:H_positive}
H(t,r,u(t,r))\geq 0, \quad \text{if } r-t\geq \sigma_n t.
\end{equation}
\\
For small fixed $\delta>0$, we define the following blow-up set
\begin{equation}
\label{eq:sigma}
\Sigma_\delta =\Big\{(t,r)\in(0,\infty)^2: r-t\geq \max\left\{\sigma_n t,\delta\right\}\Big\},
\end{equation}
where $\sigma_n$ is the constant given in Lemma \ref{lemma:Tdeltam}; it depends on the space dimension, 
in particular from the different representations of the  solution to the  free wave equation in odd and even space dimension.

Let $u^0=u^0(t,x)$ be the solution to the linear Cauchy problem \eqref{eq:linear_radial} with initial data $\psi(r)=u_1(r)$.
Due to \eqref{eq:H_positive}, we can use formulas \eqref{eq:um} and \eqref{eq:um0}. Since \eqref{eq:data_assumption} holds, 
we may estimate 
\begin{equation*}
u(t,r)\geq \epsilon u^0(t,r)\geq \frac{\epsilon}{8r^m}\int_{r-t}^{r+t}\lambda^m e^{U(\lambda)}v_1(\lambda)d\lambda\geq \frac{M\epsilon}{8r^m}\int_{r-t}^{r+t}\frac{\lambda^m e^{U(\lambda)}}{(1+\lambda)^{\alpha+1}}d\lambda,\,
\end{equation*}
for any $(t,r)\in\Sigma_\delta.$
Being $r-t>\delta,$ we derive
\begin{align*}
u(t,r)&\geq 
\frac{M\e}{8r^m}\left(\frac{\delta}{1+\delta}\right)^{\alpha+1}\int_{r-t}^{r+t}\lambda^{m-(\alpha+1)} e^{U(\lambda)}d\lambda.
\end{align*}
We put
\begin{equation}
\label{eq:U_min_max}
    \begin{aligned}
\bar U(t,r)&:=\min\left\{U(\xi): r-t\leq\xi\leq r+t \right\},\\
\bbar U(t,r)&:=\max\left\{U(\xi): r-t\leq\xi\leq r+t \right\}.
\end{aligned}
\end{equation}
Hence, we get
\begin{align*}
u(t,r)&\geq \frac{M \epsilon}{8r^m}\left( \frac{\delta}{1+\delta}\right)^{\alpha+1}(r+t)^{-(\alpha+1)} e^{\bar{U}(t,r)}\int_{r-t}^{r+t}\lambda^{m} d\lambda \\
& \geq \frac{M\e}{8}\left( \frac{\delta}{1+\delta}\right)^{\alpha+1}
\frac{2t(r-t)^{m} e^{\bar U(t,r)}}{r^m(r+t)^{\alpha+1}}.
\end{align*}
Thus,  we conclude
\begin{equation}\label{eq:estimate_from_below}
u(t,r)\geq \frac{C_0t^{m+1}{e^{\bar U(t,r)}}}{r^m(r+t)^{\alpha+1}}\,,
\end{equation}
where 
\begin{equation}\label{eq:Czero}
C_0=\e\frac{\sigma_n^m M}{4}\left(\frac{\delta}{1+\delta}\right)^{\alpha+1}>0.
\end{equation}

Let us refine \eqref{eq:estimate_from_below} by using \eqref{eq:F_up} and an iterative argument. Let us suppose that the solution $u$ to \eqref{eq:radial_problem} satisfies an estimate of the form 
\begin{equation}
\label{eq:step0}
u(t,r)\geq \frac{Ct^a e^{d\,\bar U(t,\lambda)}} {r^m(r+t)^b}, \quad \text{for any } (t,r)\in \Sigma_\delta,
\end{equation}
where $a$, $b$, $d$ and $C$ are positive constant. In particular, \eqref{eq:step0} holds for $a=m+1$, $b=\alpha+1$, $d=1$ and $C=C_0$. \\
Since the initial data $u_1$ is positive, from \eqref{eq:um0} we know that the solution $u^0$ to the linear problem associated to \eqref{eq:radial_problem} is positive; we deduce that the solution $u$ to \eqref{eq:radial_problem} also satisfies 
\begin{equation*}
u(t,r)\geq \frac{1}{8r^m}\int_0^t d\tau \int_{r-t+\tau}^{r+t-\tau}\lambda^m H(\tau,\lambda, u(\tau,\lambda))d\lambda.
\end{equation*}
Recalling $h(t,r)\geq0$ and \eqref{eq:H_positive}, we have 
\[ H(\tau,\lambda, u(\tau,\lambda))\geq e^{G(\tau)}e^{U(\lambda)}F\Big(e^{-G(\tau)} e^{-U(\lambda)}u(\tau,\lambda)\Big)\]
for any $(\tau, \lambda)\in \Sigma_\delta$. Since $F(s)\geq |s|^p$ and estimate \eqref{eq:step0} holds, we derive 
\[ H(\tau,\lambda, u)\geq \frac{C^p\tau^{pa}e^{pd\bar U(\tau,\lambda)}}{\lambda^{mp}(\lambda + \tau)^{pb}}e^{-(p-1)G(\tau)}e^{-(p-1)U(\lambda)}.\]
We may estimate
\begin{equation}
\label{eq:mainstepproof}
\begin{aligned}
u(t,r)&\geq \frac{C^p}{8r^m}\int_0^t d\tau \int_{r-t+\tau}^{r+t-\tau}\frac{\tau^{pa}(e^{pd \bar U(\tau,\lambda)})}{\lambda^{m(p-1)}(\lambda+\tau)^{pb}}e^{-(p-1)G(\tau)}e^{-(p-1)U(\lambda)} d\lambda\\
&\geq \frac{C^p}{8r^m}e^{pd \bar U(r,t)}e^{-(p-1)\bbar U(t,r)}\int_0^t \tau^{pa} e^{-(p-1)G(\tau)}  d\tau \int_{r-t+\tau}^{r+t-\tau}\frac{1}{\lambda^{m(p-1)}(\lambda+\tau)^{pb}}  d\lambda.\\
\end{aligned}
\end{equation}
We put \begin{equation}
\bbar{G}(t):=\max\left\{G(\xi): 0\leq\xi\leq t\right\};
\end{equation}
then, we can write 
\begin{equation}
\begin{aligned}
u(t,r)& \geq\frac{C^p}{8r^m(r+t)^{pb+m(p-1)}}e^{pd \bar U(r,t)}e^{-(p-1)\bbar U(t,r)} e^{-(p-1)\bbar{G}(t)}\int_0^t \tau^{pa}d\tau \int_{r-t+\tau}^{r+t-\tau}d\lambda.
%& \geq \frac{C^p e^{pdU(r\pm \tau)}}{8r^m(r+t)^{pb+m(p-1)}} e^{-(p-1)\bbar {G}(t)}e^{-(p-1)U(r \mp t)} \int_0^t (t-\tau)\tau^{pa}d\tau.
\end{aligned}
\end{equation}
Applying integration by parts, we easily obtain 
\begin{equation*}
\int_0^t (t-\tau)\tau^{pa}d\tau
\geq 
\frac{t^{pa+2}}{(pa+1)(pa+2)}.
\end{equation*}

Let $(t,r)\in \Sigma_\delta$, from \eqref{eq:step0} and \eqref{eq:mainstepproof} we can conclude
\begin{equation}
\label{eq:step1}
u(t,r)\geq \frac{C^*t^{a^*}}{r^m(r+t)^{b^*}}e^{-\ell^*\bbar{G}(t)} e^{d^*\bar{U}(t,r)} e^{-\ell^* \bbar{U}(t,r)}, \quad \text{for any }(t,r)\in\Sigma_\delta,
\end{equation}
with 
\begin{equation*}
a^*=pa+2, \hspace{2em} b^*=pb+m(p-1), \hspace{2em} d^*= pd, \hspace{2em} \ell^*=(p-1),  \hspace{2em} C^*=\frac{C^p}{8(pa+2)^2}.
\end{equation*}
Having in mind to apply an iteration procedure, we apply a third step in order to better understand the influence of the exponential terms.

Since $u$ satisfies \eqref{eq:step1} we can estimate
\begin{align*}
H(\tau,\lambda, u)&\geq \frac{C^p\tau^{pa^*}}{\lambda^{mp}(\lambda+\tau)^{pb^*}}e^{-p\ell^*\bbar G(\tau)} e^{pd^*\bar{U}(t,r)}e^{-p\ell^* \bbar U(\tau,\lambda)}e^{-(p-1)\bbar G(\tau)} e^{-(p-1)U(\lambda)}\\
& \geq \frac{C^p\tau^{pa^*}}{\lambda^{mp}(\lambda+\tau)^{pb^*}}e^{-(p\ell^*+p-1)\bbar G(\tau)}e^{pd^*\bar{U}(t,r)}e^{-(p\ell^*+p-1)\bbar U(\tau,\lambda)},
\end{align*}
for any $(\tau, \lambda)\in \Sigma_\delta$.
Then, similarly to the previous step, we can derive
\begin{align}
\label{eq:step2}
u(t,r)\geq \frac{C_1^*t^{pa^*+2}}{r^m(r+t)^{pb^*+m(p-1)}}e^{-(p\ell^*+p-1)\bbar G(t)}e^{pd^*\bar U(t,r)}e^{-(p\ell^*+p-1) \bbar U(t, r)},\quad  
\end{align}
for any $(t,r)\in\Sigma_\delta$,
where $C^*_1:=(C^{*})^p/(8(pa^*+2)^2)$.

Let us define the sequences $\{a_k\},\,\{b_k\},\,\{d_k\},\,\{\ell_k\}, \,\{C_k\}$ for $k\in \N$ by 
\begin{equation}\label{eq:sequences}
\begin{array}{ll}
a_{k+1}=pa_k+2, & a_1=m+1,\\
b_{k+1}=pb_k+m(p-1),  & b_1=\alpha+1,\\
d_{k+1}=pd_k, & d_1= 1,\\
 \ell_{k+1}=p\ell_k+p-1, & \ell_1= 0,\\
C_{k+1}=\frac{(C_k/2)^p}{2(pa_k+2)^2}, & C_1=C_0,
\end{array}
\end{equation}
where $C_0$ is defined by \eqref{eq:Czero}. Hence, we have
\begin{align}
&a_{k+1}=p^k\left(m+1+\frac{2}{p-1}\right)-\frac{2}{p-1},\label{eq:sequ1}\\
&b_{k+1}=p^k(\alpha+1+m)-m,\label{eq:sequ2} \\
& d_{k+1}=p^{k-1}, \\
& \ell_{k+1}= p^k-1, \\
& \label{eq:Ck+1first} C_{k+1}\geq K\frac{C_k^p}{p^{2k}},
\end{align}
for some constant $K=K(p,m)>0$ independent of $k$.

Following the same procedure as in \eqref{eq:step0}, \eqref{eq:step1} and \eqref{eq:step2} one can prove for any $k\geq 1$
\begin{equation}
\label{eq:generalestimate}
u(t,r)\geq \frac{C_kt^{a_{k}}}{r^m(r+t)^{b_k}}e^{-\ell_k \bbar G(t)}e^{d_k\bar{U}(t,r)}e^{-\ell_k \bbar U(t,r)}, \quad \text{for any }(t,r)\in\Sigma_\delta,
\end{equation}

The relation \eqref{eq:Ck+1first} implies  that for any $k\geq 1$ it holds
\begin{align}
&C_{k+1}\geq \exp\left(p^{k}\left(\log(C_0)-S_p(k)\right)\right),
\label{eq:sequ3}\\
&S_p(k)=\displaystyle{\Sigma_{j=0}^k} \sigma_j,\\& \sigma_0=0 \text{ and }  \sigma_j=\frac{j\log(p^2)-\log K}{p^{j}} \text{ for } j\geq 1.
\end{align}
We note that $\sigma_j>0$ for sufficiently large $j$, and $S_p(k)\to S_{p,K}$ for $k\to +\infty$ for some positive constant $S_{p,K}$. Thus, 
\begin{equation}
\label{eq:Ck+1}
C_{k+1}\geq \exp(p^{k}(\log (C_0)-S_{p,K})),
\end{equation}
for sufficiently large $k$.
Therefore, by \eqref{eq:generalestimate} we obtain 
\begin{equation}
\label{eq:finaluestimate}
u(t,r)\geq \frac{(r+t)^m e^{\bbar G(t)} e^{\bbar U(t,r)}}{r^m t^{\frac{2}{p-1}}}\exp(p^k J(t,r)),
\end{equation}
where 
\begin{align*}
J(t,r):=&\log(C_0)-S_{p,K}+\Big(m+1+\frac{2}{p-1}\Big)\log t-(\alpha+1+m)\log(r+t)\\
&
+\frac{1}{p}\bar U(t,r)-\bbar G(t)-\bbar U(t,r).
\end{align*}
Thus, if we prove that there exists $(t_0,r_0)\in \Sigma_\delta $ such that $J(t_0,r_0)>0$, then we can conclude that the solution to \eqref{eq:radial_problem} blows up in finite time, in fact
\[ 
u(t_0,r_0)\to \infty  \text{ for } k\to \infty. 
\]
It suffices to restrict on the line $r=r(t)=t\left(1+\sigma_n\right)$. Thus, we look for a suitable $t_0>0$ such that $J(t_0, r(t_0))>0;$ this is equivalent to 
\begin{equation}
\begin{aligned}
\label{eq:final-inequality}
   &\left(\frac{2}{p-1}-\alpha\right)\log(t)-\bbar{G}(t)+\frac{1}{p} \bar U(t,t(1+\sigma_n))- \bbar U(t,t(1+\sigma_n))\\  &\hspace{18em}>\log\left(\frac{e^{S_{p,K}}}{C_0}\left(2+\sigma_n\right)^{\alpha+1+m}\right).
\end{aligned}
\end{equation}
where $\bar U$ and $\bbar U$ are defined in \eqref{eq:U_min_max}.
Since condition \eqref{eq:interaction} holds, then inequality \eqref{eq:final-inequality} is satisfied for $t$ sufficiently large. This ensures that the solution to \eqref{eq:covariant_wave_equation*} (with $j=0$) blows up in finite time.

\begin{proofof}{Theorem \ref{th.main.derivative}}
Due to the radial assumptions, for $j=1$ we rewrite \eqref{eq:wave_equation*} as
\begin{equation}
\begin{cases}
u_{tt}-u_{rr}-\frac{n-1}{r}u_r=H(t,r,u, \partial_t u), \quad &t\geq0,\,r>0,\\
u(0,r)=0,  \\
u_t(0,r)=\varepsilon e^{U(r)} v_1(r),
\end{cases}
\end{equation}
where
\begin{equation} 
H(t,r,u,\partial_t u)=h(t,r)u+e^{G(t)} e^{U(r)}F\Big(e^{-G(t)} e^{-U(r)}\partial_t u\Big).
\end{equation}
For a small fixed $\delta>0$, we define
\begin{equation}
\label{eq:sigma_derivative}
\Sigma_\delta =\Big\{(t,r)\in(0,\infty)^2: r-t\geq \max\left\{\sigma_n t,\delta\right\}\Big\},
\end{equation}
where $\sigma_n$ is the constant given in Lemma \ref{lemma:Tdeltam}.

We follow straightforward calculations as in the proof of Theorem \ref{th.main} to prove that the solution $u$ to \eqref{eq:covariant_wave_equation*} with $j=1$ satisfies 
\begin{equation}
\label{eq:step0_derivative}
u(t,r)\geq \frac{Ct^a e^{d\,\bar U(t,\lambda)}} {r^m(r+t)^b}, \quad \text{for any } (t,r)\in \Sigma_\delta,
\end{equation}
where $a=m+1$, $b=\alpha+1$, $d=1$ and $C=C_0$ given by
\begin{equation}
C_0=\e\frac{\sigma_n^m M}{4}\left(\frac{\delta}{1+\delta}\right)^{\alpha+1}>0;
\end{equation}
moreover, 
\begin{equation}
    \begin{aligned}
\bar U(t,r):=\min\left\{U(\xi): r-t\leq\xi\leq r+t \right\}.\\
%\bbar U(t,r)&:=\max\left\{U(\xi): r-t\leq\xi\leq r+t \right\}.
\end{aligned}
\end{equation} 
Now we can apply Lemma \ref{lemma:Tdeltam} to the solution of \eqref{eq:radial_problem} since from \eqref{eq:H_positive} we know $H\geq 0$ for any $(t,r)\in \Sigma_\delta$. Moreover, we can also apply Lemma \ref{lem:comparison} to the solution $u^0$ of \eqref{eq:linear_radial} with $\psi=\eps e^{U(r)}v_1(r)>0$. Hence,
\begin{equation*}
u(t,r)\geq \frac{1}{8r^m}\int_0^t d\tau \int_{r-t+\tau}^{r+t-\tau}\lambda^m H(\tau,\lambda, u(\tau,\lambda))d\lambda.
\end{equation*}
In turn this gives $u(t,r)\geq 0$ and, recalling $h(t,r)\geq0$
\[ H(\tau,\lambda, u, \partial_t u)\geq e^{-(p-1)G(\tau)}e^{-(p-1)U(\lambda)}|\partial_t u(\tau,\lambda)|^p, \]
and then 
\begin{equation*}
u(t,r)\geq \frac{1}{8r^m}\int_0^t d\tau \int_{r-t+\tau}^{r+t-\tau}\lambda^m e^{-(p-1)G(\tau)}e^{-(p-1)U(\lambda)}|\partial_t u(\tau,\lambda)|^p d\lambda.    
\end{equation*}
Exchanging the order of the integrals we get
\begin{equation}
\label{eq:step1_derivative}
\begin{aligned}
u(t,r)\geq \frac{1}{8r^m}\int_{r-t}^r d\lambda &\int_0^{\lambda-r+t} \lambda^m e^{-(p-1)G(\tau)}e^{-(p-1)U(\lambda)}|\partial_t u(\tau,\lambda)|^p d\tau \\
&+ \frac{1}{8r^m}\int_r^{r+t} d\lambda \int_0^{r+t-\lambda} \lambda^m e^{-(p-1)G(\tau)}e^{-(p-1)U(\lambda)}|\partial_t u(\tau,\lambda)|^p d\tau \\
& \hspace{-72pt} \geq \frac{1}{8r^m}\int_r^{r+t} d\lambda \int_0^{r+t-\lambda} \lambda^m e^{-(p-1)G(\tau)}e^{-(p-1)U(\lambda)}|\partial_t u(\tau,\lambda)|^p d\tau. 
\end{aligned}    
\end{equation}
Since $u(0,\lambda)=0$, applying H\"older inequality we may estimate 
\begin{align*}
|u(r+t-\lambda,\lambda)|^p \leq e^{(p-1)G(r+t-\lambda)}\int_0^{r+t-\lambda} e^{-(p-1)G(\tau)}|\partial_t u(\tau,\lambda)|^pd\tau.
\end{align*}
From \eqref{eq:step0_derivative} and \eqref{eq:step1_derivative} this allows to derive:
\begin{equation}
\begin{aligned}
u(t,r)&\geq \frac{1}{8r^m}\int_r^{r+t} \lambda^m e^{-(p-1)U(\lambda)}e^{-(p-1)G(r+t-\lambda)} |u(r+t-\lambda,\lambda)|^pd\lambda \\
&\geq   \frac{C^p}{8r^m}\int_r^{r+t} \lambda^m e^{-(p-1)U(\lambda)}e^{-(p-1)G(r+t-\lambda)}\frac{(r+t-\lambda)^{ap} e^{dp\,\bar U(r+t-\lambda,\lambda)}} {\lambda^{mp}(r+t)^{bp}}d\lambda.
\end{aligned}
\end{equation}
Let us define
\begin{equation}
\begin{aligned}
\bbar U(t,r)&:=\max\left\{U(\xi): r-t\leq\xi\leq r+t \right\},\\
\bbar{G}(t)&:=\max\left\{G(\xi): 0\leq\xi\leq t\right\}.
\end{aligned}   
\end{equation}
Being $\lambda\in [r,r+t]$ it is easy to note that $\bar U(r+t-\lambda,\lambda)>\bar U(t,r)$ and $U(\lambda)<\bbar U(t,r)$; moreover, it holds $G(r+t-\lambda)<\bbar{G}(t)$ and $\lambda^{-m(p-1)}>(r+t)^{-m(p-1)}$.
Thus, we get
\begin{equation}
\begin{aligned} 
\label{eq:step2_derivative}
u(t,r)&\geq \frac{C^p e^{-(p-1)\bbar{U}(t,r)}e^{dp \bar{U}(t,r)}e^{-(p-1)\bbar{G}(t)}} {8r^m(r+t)^{bp+m(p-1)}}\int_r^{r+t} (r+t-\lambda)^{ap}d\lambda\\
&= \frac{C^p e^{-(p-1)\bbar{U}(t,r)}e^{dp \bar{U}(t,r)}e^{-(p-1)\bbar{G}(t)}t^{ap+1}}{8(ap+1)r^m(r+t)^{bp+m(p-1)}}.
\end{aligned}
\end{equation}
We can conclude that for any  $(t,r)\in \Sigma_\delta$ the solution $u$ satisfies
\begin{equation}
u(t,r)\geq \frac{C^*t^{a^*}}{r^m(r+t)^{b^*}}e^{-\ell^*\bbar{G}(t)} e^{d^*\bar{U}(t,r)} e^{-\ell^* \bbar{U}(t,r)}, \quad \text{for any }(t,r)\in\Sigma_\delta,
\end{equation}
with 
\begin{equation*}
a^*=pa+1, \hspace{2em} b^*=pb+m(p-1), \hspace{2em} d^*= pd, \hspace{2em} \ell^*=(p-1),  \hspace{2em} C^*=\frac{C^p}{8(pa+1)}.
\end{equation*}
Similarly, by induction one can prove that for any $k\geq 1$ it holds
\begin{equation}
\label{eq:generalestimate_derivative}
u(t,r)\geq \frac{C_kt^{a_{k}}}{r^m(r+t)^{b_k}}e^{-\ell_k \bbar G(t)}e^{d_k\bar{U}(t,r)}e^{-\ell_k \bbar U(t,r)}, \quad \text{for any }(t,r)\in\Sigma_\delta,
\end{equation}
where
\begin{align}
&a_{k+1}=p^k\left(m+1+\frac{1}{p-1}\right)-\frac{1}{p-1},\\
&b_{k+1}=p^k(\alpha+1+m)-m, \\
& d_{k+1}=p^{k-1}, \\
& \ell_{k+1}= p^k-1, \\
& C_{k+1}\geq K\frac{C_k^p}{p^{2k}},
\end{align}
for some constant $K=K(p,m)>0$ independent of $k$.
Following straightforward calculations as in the proof of Theorem \ref{th.main}, as a consequence of  \eqref{eq:generalestimate_derivative}, one can prove that there exists a positive constant $\tilde S_{p,K}$ such that
\begin{equation}
\label{eq:finaluestimate_derivative}
u(t,r)\geq \frac{(r+t)^m e^{\bbar G(t)} e^{\bbar U(t,r)}}{r^m t^{\frac{1}{p-1}}}\exp(p^k \tilde J(t,r)),
\end{equation}
for any $(t,r)\in \Sigma_\delta$, where 
\begin{align*}
\tilde J(t,r):=&\log(C_0)-\tilde S_{p,K}+\Big(m+1+\frac{1}{p-1}\Big)\log t-(\alpha+1+m)\log(r+t)\\
&
+\frac{1}{p}\bar U(t,r)-\bbar G(t)-\bbar U(t,r).
\end{align*}
The desired blow up result follows if we prove the existence of $(t_0,r_0)\in \Sigma_\delta $ such that $\tilde J(t_0,r_0)>0$.
Restricting to the line $r=r(t)=t\left(1+\sigma_n\right)$ the condition $\tilde J(t,r)>0$ is equivalent to 
\begin{equation}
\begin{aligned}
\label{eq:final-inequality-derivative}
   &\left(\frac{1}{p-1}-\alpha\right)\log(t)-\bbar{G}(t)+\frac{1}{p} \bar U(t,t(1+\sigma_n))- \bbar U(t,t(1+\sigma_n))\\  &\hspace{18em}>\log\left(\frac{e^{S_{p,K}}}{C_0}\left(2+\sigma_n\right)^{\alpha+1+m}\right).
\end{aligned}
\end{equation}
Condition \eqref{eq:interaction_derivative} guarantees that inequality \eqref{eq:final-inequality-derivative} is satisfied for $t$ sufficiently large. In this case the solution $u$ to problem \eqref{eq:covariant_wave_equation*} (with $j=1$) blows up in finite time.
\end{proofof}

\section{Lifespan estimates for the solution}
\label{sec:lifespan}
We can estimate the lifespan only for certain specific choices of the functions $A$ and $A_0$ in \eqref{cov_space} and, respectively, \eqref{cov_time}. Following Section \ref{sec:examples} we give the following result whose assumptions correspond to Example \ref{ex:scaleinvariant}, \ref{ex:3} and \ref{ex:4}. 
\begin{theorem}\label{thm:lifespan}
Let $n\geq 2$ and $\sigma_n>0$ given by Lemma \ref{lemma:Tdeltam}. Let $v_1$ be a radial smooth function satisfying 
\begin{equation}
    v_1(r)\geq \frac{M}{(1+r)^{\alpha+1}}, \quad \forall r>0,
\end{equation}
for some $\alpha>-1$ and $M>0$. Let $h\geq 0$ and $p>1$. Let us consider the Cauchy problem
\begin{equation}
\label{eq:CP_lifespan}
\begin{cases}
\tilde\partial_{tt}v(t,x)- \tilde\Delta v(t,x)=h(t,x)v(t,x)+|\tilde\partial_t^j v(t,x)|^p, \quad (t,x)\in [0,\infty)\times \R^n, \\
v(0,x)=0,&\\
v_t(0,x)= \varepsilon v_1(x),&
\end{cases}
\end{equation}
where $\tilde \partial_t=\partial_t+A_0(t)$ and $\tilde\partial_{x_i}=\partial_{x_i}+\partial_{x_i}U$ with  $A_0\geq 0$ and $U$ smooth, radial function; assume that
\begin{equation}\label{asmp_thm_lifespan}
\lim_{s\to+\infty}sA_0(s)=\gamma \quad \text{ and } \quad \lim_{\xi\to+\infty}\frac{U(\xi)}{\log(\xi)}=\ell
\end{equation}
for some $\gamma,\ell \geq 0$. 
\\ Let $p_{\ell,j}(\alpha+\gamma)$ be the positive root of the identity
\begin{equation}
(\alpha+\gamma+\ell)p^2-(\alpha+\gamma+2\ell+2-j)p+\ell=0.
\end{equation}
If $\alpha+\gamma+\ell>0$ and $1<p<p_{\ell,j}(\alpha+\gamma)$ then the classical solution $v$ of \eqref{eq:CP_lifespan}
blows up in finite time.
The same holds true for any $p>1$ if $\alpha+\gamma+\ell\leq 0$. \\
In particular, there exists $\eps_0>0$ such that for any $0<\eps<\eps_0$ the finite lifespan of $v$ satisfies
\begin{equation*}
T_\eps\leq  C \eps^{-\left(\frac{2-j}{p-1}-\alpha-\gamma-\ell\left(1-\frac{1}{p}\right)\right)^{-1}}
\end{equation*}
for some constant $C>0$, independent of $\eps$.
\end{theorem}

\noindent 

We firstly consider Example \ref{ex:scaleinvariant} in order to compare our result with the one obtained in \cite{CGL2021}. Here, we have $A_0(t)=\frac{\mu}{2(1+t)}$ and $U(x)=\frac{\eta}{2}\log\langle x \rangle$ that means $\gamma=\mu/2$ and $\ell=\eta/2$. We extend the result obtained in \cite{CGL2021} to the case $\eta\neq0.$ \\
We follow the proof of Theorem \ref{th.main} (if $j=0$) or Theorem \ref{th.main.derivative} (if $j=1$) until \eqref{eq:final-inequality} or, respectively, \eqref{eq:final-inequality-derivative}.  We rewrite \eqref{eq:Czero} as $C_0:=\varepsilon \tilde C_0$ where
\begin{equation*}
\tilde C_0=\frac{\sigma_n^m M}{4}\left(\frac{\delta}{1+\delta}\right)^{\alpha+1}.
\end{equation*}
The solution to \eqref{eq:CP_lifespan} blows up in finite time if there exists $t>0$ such that
\begin{equation}
\begin{aligned}
%\label{eq:final-inequality}
   &\left(\frac{2-j}{p-1}-\alpha\right)\log(t)-\frac{\mu}{2}\log(1+t)+ \frac{\eta}{2p}\log\langle \sigma_n t \rangle- \frac{\eta}{2}\log\langle (2+\sigma_n) t \rangle\\ &\hspace{18em}>\log\left(\frac{e^{S_{p,K}}}{\e \tilde C_0}(2+\sigma_n)^{\alpha+1+m}\right).
\end{aligned}
\end{equation}
For any $t>1,$ this inequality is satisfied if  
\begin{equation}
\left(\frac{2-j}{p-1}-\alpha-\frac{\mu}{2}-\frac{\eta}{2}\left(1-\frac{1}{p}\right)\right)\log(t)>\log\left(\frac{e^{S_{p,K}}}{\e \tilde C_0}2^\frac{\mu}{2}\sigma_n^{-\frac{\eta}{2p}}(2+\sigma_n)^{\frac{\eta}{2}+\alpha+1+m}\right).
\end{equation}
Focusing on the size of the initial data, we get
\begin{equation}
t>C\varepsilon^{-\left(\frac{2-j}{p-1}-\alpha-\frac{\mu}{2}-\frac{\eta}{2}\left(1-\frac{1}{p}\right)\right)^{-1}}.
\end{equation}
We need 
\[ \frac{2-j}{p-1}-\alpha-\frac{\mu}{2}-\frac{\eta}{2}\left(1-\frac{1}{p}\right)>0.\]
In particular, this is satisfied if $p_{\ell,j}\left(\alpha+\frac{\mu}{2}\right)$ is the positive root of the identity
\begin{equation}
(\alpha+\eta/2+\mu/2)p^2-(\alpha+\mu/2+2-j+\eta)p+\eta/2=0,
\end{equation}
for $\alpha+\eta/2+\mu/2>0$ and $1<p<p_{\frac{\eta}{2},j}(\alpha+\mu/2)$, or $p>1$ and $\alpha+\eta/2+\mu/2<0$ then the classical solution $v$ of \eqref{eq:CP_lifespan}
blows up in finite time.

In particular, $C>0$ is independent of $\eps$ indeed
\begin{align}
C=\left(\left(\frac{1+\delta}{\delta}\right)^{\alpha+1}\frac{4 e^{S_{p,K}}}{M}2^\frac{\mu}{2}\sigma_n^{-(m+\frac{\eta}{2p})}(2+\sigma_n)^{\frac{\eta}{2}+\alpha+1+m}\right)^{\left(\frac{2-j}{p-1}-\alpha-\frac{\mu}{2}-\frac{\eta}{2}\left(1-\frac{1}{p}\right)\right)^{-1}}.
\end{align}

\begin{proof}
Now we prove Theorem \ref{thm:lifespan} that means we discuss Examples~\ref{ex:3} and\ref{ex:4}.\\ 
Due to \eqref{asmp_thm_lifespan} there exists $\bar t>0$ such that  we can write inequality \eqref{eq:final-inequality} (if $j=0$) or \eqref{eq:final-inequality-derivative} (if $j=1$) as
\begin{equation}
\left(\frac{2-j}{p-1}-\left(\alpha+\gamma\right)-\ell\left(1-\frac{1}{p}\right)\right)\log(t)>\log\left(\frac{\bar{C}}{\eps}\right),
\end{equation}
where $\bar{C}>0$ is independent of $\eps$ and $\bar t$. This is satisfied for
\begin{equation}
\label{eq:lifespan_ex4}
t>C\varepsilon^{-\left(\frac{2-j}{p-1}-\alpha-\gamma-\ell\left(1-\frac{1}{p}\right)\right)^{-1}}.
\end{equation}
The exponent of $\eps$ is negative provided $p>1$ if $\alpha+\gamma+\ell\leq 0$, or $1<p<p_{\ell,j}(\alpha+\gamma)$ if $\alpha+\gamma+\ell>0.$ Let us prove that $C>0$ is independent of $\varepsilon$ and $\bar t.$ 
For any $\eps_0>0$ arbitrarily small and $t>0$ sufficiently large, inequality \eqref{eq:final-inequality} holds if
\begin{equation}
\label{eq:final_inequality_lifespan_examples2-4}
\begin{aligned}
   &\left(\frac{2-j}{p-1}-\alpha-\gamma-\eps_0\right)\log(t)+\frac{1}{p} \min_{[\sigma_n t,(2+\sigma_n)t]}(\ell-\eps_0)\log(\xi)- \max_{[\sigma_n t,(2+\sigma_n)t]}(\ell+\eps_0)\log(\xi)\\ &\hspace{18em}>\log\left(\frac{e^{S_{p,K}}}{\eps \tilde C_0}\left(2+\sigma_n\right)^{\alpha+1+m}\right).
\end{aligned}
\end{equation}
%
%Being $\eps_0>0$ arbitrarily small, if $\ell=0$ inequality \eqref{eq:final_inequality_lifespan_examples2-4} implies the blow up of the solution when $1<p<1+2/(\alpha+\gamma)$ and 
%
%\begin{equation}
%t>C\varepsilon^{-\left(\frac{2}{p-1}-\alpha-\gamma\right)^{-1}},
%\end{equation}
%
%where 
%
%\begin{align}
%C=\left(\left(\frac{1+\delta}{\delta}\right)^{\alpha+1}\frac{4 e^{S_{p,K}}}{M}\sigma_n^{-m}(2+\sigma_n)^{\alpha+1+m}\right)^{\left(\frac{2}{p-1}-\alpha-\gamma\right)^{-1}}.
%\end{align}
%
%If $\ell>0$ then \eqref{eq:final_inequality_lifespan_examples2-4} holds if 
%
%\begin{equation}
%\begin{aligned}
%   &\left(\frac{2}{p-1}-\alpha-\gamma-\ell\left(1-\frac{1}{p}\right)\right)\log(t)+\frac{\ell}{p} \log(\sigma_n)-\ell\log(2+\sigma_n)\\ &\hspace{18em}>\log\left(\frac{e^{S_{p,K}}}{\eps \tilde C_0}\left(2+\sigma_n\right)^{\alpha+1+m}\right),
%\end{aligned}
%
%\end{equation}
%
%which is satisfied for
%
%\begin{equation}
%\label{eq:lifespan_ex4}
%t>C\varepsilon^{-\left(\frac{2}{p-1}-\alpha-\gamma-\ell\left(1-\frac{1}{p}\right)\right)^{-1}},
%\end{equation}
%
%provided $p>1$ if $\alpha+\gamma+\ell\leq 0$, or $1<p<p_\ell(\alpha+\gamma)$ if $\alpha+\gamma+\ell>0$; 
Thus, for $\ell>0,$ in \eqref{eq:lifespan_ex4} we can write  
\begin{align}
C=\left(\left(\frac{1+\delta}{\delta}\right)^{\alpha+1}\frac{4 e^{S_{p,K}}}{M}\sigma_n^{-(m+\frac{\ell}{p})}(2+\sigma_n)^{\ell+\alpha+1+m}\right)^{\left(\frac{2-\red{j}}{p-1}-\alpha-\gamma-\ell\left(1-\frac{1}{p}\right)\right)^{-1}}.
\end{align}
Similarly, if $\ell<0$ from \eqref{eq:final_inequality_lifespan_examples2-4} we derive \eqref{eq:lifespan_ex4} with 
\begin{align}
C=\left(\left(\frac{1+\delta}{\delta}\right)^{\alpha+1}\frac{4 e^{S_{p,K}}}{M}\sigma_n^{-m+\ell}(2+\sigma_n)^{\alpha+1+m-\frac{\ell}{p}}\right)^{\left(\frac{2\red{-j}}{p-1}-\alpha-\gamma-\ell\left(1-\frac{1}{p}\right)\right)^{-1}}.
\end{align}
This gives us the statement.
\end{proof}

\section*{Acknowledgements} 
The first author has been partially supported by INdAM GNAMPA Project “Analysis and Control of Evolutionary Models with Nonlocal Phenomena'', Grant Code CUP E5324001950001 and by University of L'Aquila Ateneo Project “Leggi di conservazione con termini nonlocali e applicazioni al traffico veicolare''. The second and third authors have been partially supported by INdAM GNAMPA Project “Modelli locali e non-locali con perturbazioni non-lineari'', Grant Code CUP E55F22000270001. The third Author has been supported by PRIN 2022 “Anomalies in partial differential equations and applications” CUP H53C24000820006 and NextGenerationEU project CN00000013 -  CUP H93C22000450007.  
%%%%%%%%%%%%%%%%%%%%%%%%%%%%%%%%%%%%%%%%%%%%%%%%%%%%%%%%%%%%%%%%%%%%

%\input{AnswersCGL}

\begin{thebibliography}{abcd}
\bibitem[A]{A1986} F. Asakura, {\it Existence of a global solution to a semi-linear wave equation with slowly decreasing initial data in three space dimensions}. Communications in  Partial Differential Equations 11 (1986), 1459-1487.
\bibitem[AT]{AT1992} R. Agemi, H. Takamura, {\it The lifespan of classical solutions to nonlinear wave equations in two
space dimensions}. Hokkaido Mathematical Journal 21 (1992), 517-542.

\bibitem[CGL]{CGL2021} F. A. Chiarello, G. Girardi, S. Lucente, {\it Fujita modified exponent for scale invariant damped semilinear wave equations}.  Journal of Evolution Equations 21 (2021), 2735–2748.
%
\bibitem[DA]{DA2020} P. D'Ancona, { \it On large potential perturbations of the Schrödinger, Wave and Klein–Gordon equations}. Communications on Pure and Applied Analysis 19 (2020), 609-640. 
\bibitem[DAF]{DAF2008} P. D’Ancona, L. Fanelli, {\it Strichartz and smoothing estimates for dispersive equations with magnetic potentials}. Communications in Partial Differential Equations 33(6) (2008), 1082–1112.
%
\bibitem[DALR]{DALR2015} M. D’Abbicco, S. Lucente, M. Reissig,  {\it A shift in the Strauss exponent for semilinear wave equations with a not effective damping}. Journal of Differential Equations, 259 (2015) 5040--5073.
\bibitem[FV]{FV2009} L. Fanelli, L. Vega, {\it Magnetic virial identities, weak dispersion and Strichartz inequalities}. Mathematische Annalen  344 (2009), 249–278 . 
\bibitem[GYZZ]{GYZZ2022} X. Gao, Z.  Yin, J.  Zhang, J. Zheng, {\it Decay and Strichartz estimates in critical electromagnetic fields.}
Journal of Functional Analysis 282 (2022), 109350.
\bibitem[GLS]{GLS1997} V. Georgiev, H. Lindblad, C. D. Sogge, {\it Weighted Strichartz estimates and global existence for semilinear wave equations}. American Journal of Mathematics 119 (1997), 1291–1319.
\bibitem[GeLu]{GeorgievLucente2021} V. Georgiev, S. Lucente, {\it Quasilinear Wave equations with decaying time-potential}. In: Georgiev, V., Michelangeli, A., Scandone, R. (eds) Qualitative Properties of Dispersive PDEs. INdAM 2021. Springer INdAM Series, vol 52. Springer, Singapore  (2022). 

\bibitem[GiLu]{GirardiLucente2021} G. Girardi, S. Lucente, {\it Lifespan estimates for a special quasilinear time-dependent damped wave equation}. In: Cerejeiras, P., Reissig, M., Sabadini, I., Toft, J. (eds) Current Trends in Analysis, its Applications and Computation. Trends in Mathematics. Birkhäuser, Cham (2022). 

\bibitem[K]{K97} 
H. Kubo, {\em Slowly decaying solutions for semilinear wave equations in odd space dimensions.} Nonlinear Analysis: Theory, Methods \& Applications 28 (1997), 327-357.
	
\bibitem[KK]{K98}
H. Kubo, K. Kubota, {\em Asymptotic behaviors of radially symmetric solutions of  $\Box u=|u|^p$ for super critical values $p$ in even space dimensions.} Japanese Journal of Mathematics. New series 24 (1998), 191--256.
\bibitem[NPR]{NPR2017} W. Nunes do Nascimento, A. Palmieri, M. Reissig, {\it Semi-linear wave models with power non-linearity and scale-invariant time-dependent mass and dissipation}. Mathematische Nachrichten 290 (2017), 1779-1805.
\bibitem[PT]{PT2021} A. Palmieri, Z. Tu, {\it A blow-up result for a semilinear wave equation with scale-invariant damping and mass and nonlinearity of derivative type}.  Calculus of Variations and Partial Differential Equations  60:72 (2021).
\bibitem[T]{takamura1995} H. Takamura, {\it Blow-up for semilinear wave equations with slowly decaying data in high dimensions}. Differential Integral Equations 8 (1995), 647--661.


\end{thebibliography}
\end{document}